\documentclass[12pt]{amsart}
\usepackage{amsmath,eucal,verbatim,amsopn,amsthm,amsfonts,amssymb,mathdots,mathrsfs,esint,mathtools,enumitem,bbm}
\usepackage{amsmath,amssymb,latexsym}
\usepackage{color}
\usepackage{amscd}
\usepackage[margin=1in]{geometry}

\renewcommand\theequation{\thesection.\arabic{equation}}

\newcommand{\GL}{{\mathrm{GL}}}

\newcommand{\Ind}{{\mathrm{Ind}}}
\renewcommand{\Re}{{\mathrm{Re}}}
\newcommand{\SL}{{\mathrm{SL}}}

\newcommand{\SO}{{\mathrm{SO}}}
\newcommand{\C}{{\mathbb{C}}}
\newcommand{\Q}{{\mathbb{Q}}}

\def\diag{{\rm diag}}

\newtheorem{thm}{Theorem}[section]
\newtheorem{cor}[thm]{Corollary}
\newtheorem{lem}[thm]{Lemma}
\newtheorem{prop}[thm]{Proposition}

\newtheorem {ques/conj}[thm]{Question/Conjecture}

\newtheorem{rmk}[thm]{Remark}

\makeatletter

\newcommand{\Rmnum}[1]{\expandafter\@slowromancap\romannumeral #1@}
\makeatother

\begin{document}
\renewcommand{\theequation}{\arabic{equation}}
\numberwithin{equation}{section}

\title[The Langlands parameter: even orthogonal groups]{On The Langlands parameter of a Simple Supercuspidal Representation: Even Orthogonal Groups}
\author{Moshe Adrian}
\author{Eyal Kaplan}
\address{Adrian:  Department of Mathematics, Queens College, CUNY, Queens, NY 11367-1597}
\email{moshe.adrian@qc.cuny.edu}
\address{Kaplan: Department of Mathematics, Bar Ilan University, Ramat Gan 5290002, Israel}
\email{kaplaney@gmail.com}

\subjclass[2010]{Primary 11S37, 22E50; Secondary 11F85, 22E55}
\keywords{Simple supercuspidal, Local Langlands Conjecture, Rankin--Selberg method}

\begin{abstract}
Let $\pi$ be a simple supercuspidal representation of the split even special orthogonal group. We compute the
Rankin--Selberg $\gamma$-factors for rank $1$-twists of $\pi$ by quadratic tamely ramified characters
of $F^*$. We then use our results to determine the Langlands parameter of $\pi$ up to its restriction to the wild inertia subgroup, subject to an analogue of a work of Blondel, Henniart and Stevens for $\SO_{2l}$. In the particular case of the field $\Q_2$, we are able to describe the parameter completely.
\end{abstract}

\maketitle

\section{Introduction}\label{introduction}
Let $F$ be a local $p$-adic field of characteristic $0$ and $\SO_{2l}=\SO_{2l}(F)$ be the split special even orthogonal group of rank $l$.
In this work we compute a certain family of twisted gamma factors of an arbitrary simple supercuspidal representation $\pi$ of $\SO_{2l}$.
Our results enable us to explicitly identify the quadratic, tamely ramified $1$-dimensional summands of the Langlands parameter of $\pi$. It is expected that the rest of the parameter corresponds to a simple supercuspidal representation of $\GL_{2l-2}$, this can be verified by proving an analogue of the work of Blondel \textit{et al.} \cite{BHS17} for $\SO_{2l}$.  It can be shown that our results therefore give the Langlands parameter up to its restriction to the wild inertia subgroup, subject to such an analogue. The present work is the follow-up to \cite{Adrian2016,AdrianKaplan2018}, where the analogous computations were carried out and the theory of Rankin--Selberg integrals was applied, in order to determine the Langlands parameter of odd orthogonal groups and symplectic groups.

Let $\pi$ be a simple supercuspidal representation of $\SO_{2l}$. Throughout, supercuspidal represenations are assumed to be irreducible.
The representation $\pi$ is known to be generic, for a certain character of a maximal unipotent subgroup of $\SO_{2l}$.
Thus $\pi$ admits a local functorial lift $\Pi$ to $\GL_{2l}$, as defined by Cogdell \textit{et. al.} \cite[Proposition~7.2]{CKPS}.
In general, one can then study the Langlands parameter $\varphi_{\pi}$ of $\pi$ using an explicit local Langlands correspondence for the supercuspidal representations in the support of $\Pi$.

In our setting, when $p\ne2$, it is expected that $\varphi_{\pi}$ decomposes into $3$ summands, two of which are $1$-dimensional. Moreover, it is also expected that the $2l-2$ complement corresponds to a simple supercuspidal representation $\Pi'$ of $\GL_{2l-2}$. This representation is parameterized by a triplet consisting of a uniformizer $\varpi$ of $F$, a central character $\omega$, and a $(2l-2)$-th root of $\omega(\varpi)$.   For $p=2$ the situation is similar; now $\varphi_{\pi}$ is expected to decompose into a sum of a one $1$-dimensional representation and a
$2l-1$ complement, corresponding to a simple supercuspidal representation of $\GL_{2l-1}$.

In this work we find the quadratic, tamely ramified $1$-dimensional summands, by identifying the quadratic tamely ramified characters $\tau$ of $F^*$ such that the Rankin--Selberg $\gamma$-factor $\gamma(s,\pi\times\tau,\psi)$ (defined in \cite{Kaplan2013a,Kaplan2015}) has a pole at $s=1$. This information also yields the central character $\omega$ of $\Pi'$, and moreover we subsequently obtain the $(2l-2)$-th root of $\omega(\varpi)$ by computing $\gamma(s, \pi, \psi)$. What remains in order to fully describe the parameter is the uniformizer $\varpi$ and the proof that the complement of the computed two one-dimensional summands of $\varphi_{\pi}$ does indeed correspond to a simple supercuspidal of $\GL_{2l-2}$.   Our method does not provide these two items, but they are obtainable from an analogue of \cite{BHS17} for $\SO_{2l}$.

We turn to describe our results more precisely. The simple supercuspidal representations of $\SO_{2l}$ are parameterized by four pieces of data: a choice of uniformizer $\varpi$, an element $\alpha \in \kappa^{\times} / (\kappa^{\times})^2$ ($\kappa=\frak{o}/\frak{p}$), a sign $\epsilon = \pm 1$, and a central character $\omega$.  For convenience in later computations, we set $\gamma = -4 \alpha$.  More explicitly, let $\chi$ be an affine generic character of the pro-unipotent radical $I^+$ of an Iwahori group $I$, defined by an element $\gamma$.  The choice of uniformizer $\varpi$ in $F$ determines an element $g_{\chi}$ in $\SO_{2l}$ which normalizes $I$ and stabilizes $\chi$.  We can extend $\chi$ to $\langle g_{\chi} \rangle I^+$ in two different ways, since $g_{\chi}^2 = 1$.  Further extending $\chi$ by $\omega$ to the group $K = Z \langle g_{\chi} \rangle I^+$, and calling the new character $\chi$ again, we obtain a simple supercuspidal representation $\pi = \mathrm{Ind}_K^{\SO_{2l}} \chi$. For more details see \S~\ref{simple supercuspidals}.
The following is our main theorem, which characterizes the tamely ramified quadratic $1$-dimensional representations in the support of $\Pi$ (again, expected to be all of the $1$-dimensional summands).
\begin{thm}\label{mainmain}
Let $(\varpi, \gamma, \epsilon, \omega)$ be the parameters uniquely determined by $\pi$. Let $\tau$ be a quadratic tamely ramified character of $F^*$.
Then $\gamma(s, \pi \times \tau, \psi)$ is holomorphic and nonzero at $s=1$ if and only if
$\tau(\varpi) = -\chi(g_{\chi})\tau(\gamma)$.
\end{thm}
The theorem is proved in \S~\ref{easyside}. In fact our result is stronger: we compute $\gamma(s, \pi \times \tau, \psi)$ for arbitrary quadratic tamely ramified characters $\tau$, see Corollary~\ref{corollary:gamma for tamely ramified}.

Our main tool in this work is the $\gamma$-factor defined by the theory of Rankin--Selberg integrals for $\SO_{2l}\times\GL_n$ in \cite{Kaplan2015}, following the development of these integrals in \cite{GPS,Kaplan2010,Kaplan2012,Kaplan2013a,Kaplan2015a}. The $\gamma$-factor is essentially the proportionality factor between two integrals, related by an application of an intertwining operator. The proof of Theorem~\ref{mainmain} is based on a direct computation of this factor for $n=1$, and is among the first few applications of Rankin--Selberg integrals to results of this kind.

A subtle part of the definition of $\gamma(s,\pi\times\tau,\psi)$ is to normalize it properly, in order to obtain precise multiplicative formulas which identify this factor with the corresponding $\gamma$-factor of Shahidi (defined in \cite{Sh3}). While this normalization does not play a role in the determination of the poles, it is crucial for the computation of $\varphi_{\pi}$. Obtaining precise normalization is nontrivial. The equality between these $\gamma$-factors, in the context of Shimura-type integrals (proved in \cite{Kaplan2015}) was one of the ingredients in the work of Ichino \textit{et. al.} \cite{ILM} on the formal degree conjecture.
For other works on Rankin--Selberg integrals and their $\gamma$-factors, in the context of generic representations of classical groups,
see e.g., \cite{G,Soudry,Soudry3,GRS4,Soudry2}.

As mentioned above, this work is a follow-up to \cite{Adrian2016,AdrianKaplan2018}. The case of odd orthogonal groups \cite{Adrian2016} was a bit different in the sense that the lift $\Pi$ of $\pi$ was already expected to be simple supercuspidal. Indeed the twisted $\gamma$-factors had no poles, and their computation was sufficient to determine the Langlands parameter using, among other result, the works of M{\oe}glin \cite{Moeglin2014} and
Kaletha \cite{Kaletha2015}. For the symplectic case (\cite{AdrianKaplan2018}) and when $p \neq 2$, according to \cite{BHS17}, $\varphi_{\pi}$ decomposes into $2$ summands, one of them $1$-dimensional. This summand was again identified using an analogue of Theorem~\ref{mainmain}, and the computation of $\gamma(s,\pi,\psi)$ was then sufficient to obtain $\varphi_{\pi}$.  When $p = 2$, the symplectic case was similar to the odd orthogonal case, since the twisted $\gamma$-factors had no poles.

Simple supercuspidal representations are the supercuspidal representations of minimal nonzero depth, in the sense of Moy and Prasad \cite{MoyPrasad1994,MoyPrasad1996}. These representations were recently constructed
by Gross and Reeder \cite{GR10} and Reeder and Yu \cite{RY14}. Simple supercuspidal representations can be considered as a ``litmus paper" for statements on arbitrary supercuspidal representations. A few of the earlier studies on the construction of supercuspidal representations include \cite{Howe1977,Adler1998,BushnellKutzko1998,Stevens2008}.

The rest of this work is organized as follows. The Rankin--Selberg integral is described in \S~\ref{notation}. The simple supercuspidal representations are defined in \S~\ref{simple supercuspidals}. The computation of the $\gamma$-factor is carried out in \S~\ref{easyside}.  In \S~\ref{The Langlands parameter}, we describe the Langlands parameter up to its restriction to the wild inertia subgroup, subject to an analogue of \cite{BHS17} for $\SO_{2l}$.   Finally \S~\ref{normalization parameters} contains the computation of certain normalization factors used for the definition of the $\gamma$-factor.

\section{The groups and the Rankin--Selberg integral}\label{notation}
Let $F$ be a $p$-adic field of characteristic $0$, with a ring of integers $\mathfrak{o}$ and  maximal ideal $\mathfrak{p}$. Denote $\kappa = \frak{o} / \frak{p}$ and $q=|\kappa|$. Let $\varpi$ be a uniformizer ($|\varpi|=q^{-1}$). Fix the Haar measure $dx$ on $F$ which assigns the volume $q^{1/2}$ to $\frak{o}$, and define a measure $d^{\times}x$ on $F^{\times}$ by
$d^{\times}x=\frac{q^{1/2}}{q-1}|x|^{-1}dx$. We use the notation $\mathrm{vol}$ (resp., $\mathrm{vol}^{\times}$) to denote volumes of measurable subsets under $dx$ (resp., $d^{\times}x$), e.g. $\mathrm{vol}^{\times}(\frak{o}^{\times}) = 1$. Let $J_r\in\GL_r(F)$ denote the permutation matrix with $1$ along the anti-diagonal. For $g\in\GL_r(F)$, ${}^tg$ denotes the transpose of $g$, and $g^*=J_r{}^tg^{-1}J_r$.

Fix $\gamma\in F^*$. We define the orthogonal groups appearing in this work:
\begin{align*}
&\SO_{2l}(F)=\{g\in\SL_{2l}(F):{}^tgJ_{2l}g=J_{2l}\},\\
&\SO_{2n+1}(F)=\{g\in\SL_{2n+1}(F):{}^tgJ_{2n+1,\gamma}g=J_{2n+1,\gamma}\},\qquad J_{2n+1,\gamma}=
\left(\begin{smallmatrix} &  & J_n \\ & \gamma/2 &  \\ J_n &  & \end{smallmatrix}\right).
\end{align*}
Throughout, we identify linear groups with their $F$-points, i.e., $\SO_r=\SO_r(F)$.

Fix the Borel subgroup $B_{\SO_r}=T_{\SO_r}\ltimes U_{\SO_r}$ of upper triangular invertible matrices in $\SO_r$,
where $T_{\SO_r}$ is the diagonal torus. Denote $K_{\SO_r}=\SO_r(\mathfrak{o})$, which is a maximal compact open subgroup in $\SO_r$. Let $Z_{\SO_{r}}$ be the center of $\SO_{r}$. For a unipotent subgroup $U<\SO_r$, let $\overline{U}$ denote the opposite unipotent subgroup.

We describe the Rankin--Selberg integral for $\SO_{2l}\times\GL_1$, $l\geq2$, which will be our main tool for the computation of the $\gamma$-factor.
We follow the definitions and conventions of \cite{Kaplan2015}, where the full details of the construction for $\SO_{2l}\times\GL_n$ were given.

Let $\tau$ be a quasi-character of $F^*$. For $s \in \C$, let $V(\tau,s)$ be the space of the representation $\mathrm{Ind}_{B_{\SO_3}}^{\SO_3} (|\det|^{s-1/2}\tau)$ (normalized induction). The elements of $V(\tau,s)$ are complex-valued smooth functions $f_s$ on $\SO_3\times\GL_1$, such that
for all $a,m\in F^*$, $u\in U_{\SO_3}$ and $g\in \SO_3$,
\begin{align*}
f_s(\diag(m,1,m^{-1})ug, a) = |m|^{s}f_s(g,am)=|m|^{s}\tau(am)f_s(g,1).
\end{align*}
The right-action of $\SO_3$ on $V(\tau,s)$ is denoted $g\cdot f_s$.
A function $f_s$ is called a standard section if its restriction to $K_{\SO_3}$ is independent of $s$, and a holomorphic
section if its restriction to $K_{\SO_3}$ is a polynomial function in $q^{\mp s}$.

Let $l\geq2$ and fix a nontrivial additive character $\psi$ of $F$. Define the following non-degenerate character $\psi$ of $U_{\SO_{2l}}$ by
\begin{align}\label{def:character of U for integral}
\psi(u) = \psi (\sum_{i=1}^{l-2} u_{i,i+1} + \tfrac{1}{4} u_{l-1,l} - \gamma u_{l-1,l+1}).
\end{align}
Let $\pi$ be an irreducible $\psi^{-1}$-generic representation of $\SO_{2l}$, and denote the corresponding Whittaker model of $\pi$
by $\mathcal{W}(\pi, \psi^{-1})$.

We turn to describe the embedding of $\SO_3$ in $\SO_{2l}$. Let $Q=M\ltimes N$ be the standard parabolic subgroup of $\SO_{2l}$,
whose Levi part $M$ is isomorphic to $\GL_1\times\ldots\times \GL_1\times \SO_4$. For $l\geq3$, define a character $\psi_N$ of $N$ by
\begin{align*}
\psi_N(u)=\psi(\sum_{i=1}^{l-3}u_{i,i+1}+\tfrac{1}{4}u_{l-2,l}-\gamma u_{l-2,l+1}).
\end{align*}
The group $\SO_3$ is then embedded in $\SO_{2l}$ in the stabilizer of $\psi_N$ in $M$. When $l=2$, we embed $\SO_3$ in the subgroup of $g\in\SO_4$ such that $g(\tfrac{1}{4}e_2-\gamma e_3)=\tfrac{1}{4}e_2-\gamma e_3$, where $(e_1,\ldots,e_4)$ is the standard basis of the column space $F^4$. In coordinates, the image of $(x_{i,j})_{1\leq i,j\leq 3}\in\SO_3$ in $\SO_{2l}$ is given by
\begin{align*}
diag(I_{l-2}, \begin{pmatrix}1\\
&\tfrac14 & \tfrac14\\
&-\gamma & \gamma\\&&&1
\end{pmatrix}
\begin{pmatrix}
x_{1,1} &  & x_{1,2} & x_{1,3}\\
& 1 & & \\
x_{2,1} & & x_{2,2} & x_{2,3}\\
x_{3,1} & & x_{3,2} & x_{3,3}
\end{pmatrix}\begin{pmatrix}1\\
&2 & -\tfrac12\gamma^{-1}\\
&2 & \tfrac12\gamma^{-1}\\&&&1
\end{pmatrix}, I_{l-2}).
\end{align*}
The conjugating matrix is different from the one used in
\cite{Kaplan2015}. To explain this, let $e$ belong to the orthogonal complement of $\tfrac{1}{4}e_2-\gamma e_3$ in $F^4$ with respect to the bilinear form $(u,v)\mapsto {}^tu J_4 v$. Fixing $e$ in the span of $e_{2}$ and $e_{3}$, it belongs to the span of
$\tfrac{1}{4}e_2+\gamma e_{3}$, then $\SO_3$ is defined with respect to $(e_1,e,e_4)$ (for $l>2$, $(e_1,e_2,e_3,e_4)$ is replaced with $(e_{l-1},e_l,e_{l+1},e_{l+2})$). In \cite{Kaplan2015} $2\gamma$ was assumed to be a square (in the split case), then $e$ could be scaled to a unit vector and the Gram matrix of $(e_1,e,e_4)$ was $J_3$. Without this assumption we take here $e=\tfrac{1}{4}e_l+\gamma e_{l+1}$ and work with
$J_{3,\gamma}$. In the general case of an arbitrary $n<l$, the definition of $\SO_{2n+1}$ is then using $J_{2n+1,\gamma}$.

Also let
\begin{align*}
R^{l,1} =\left\{\begin{pmatrix}
1 & & & & \\
r & I_{l-2} & & & \\
& & I_2 & &\\
& & & I_{l-2} &\\
& & & r' & 1
\end{pmatrix} \in \SO_{2l}\right\},\qquad
w^{l,1}=\begin{pmatrix}
 &1 & & & \\
 I_{l-2} & & & \\
& & I_2 & &\\
& & & & I_{l-2} \\
& & & 1 &
\end{pmatrix} \in \SO_{2l}.
\end{align*}
We will occasionally refer to $r\in R^{l,1}$ also as a column vector in $F^{l-2}$.
Now we can define the Rankin--Selberg integral for $\pi\times\tau$: for any $W\in\mathcal{W}(\pi, \psi^{-1})$ and a holomorphic section $f_s$, the integral is defined for $\Re(s)\gg0$ by
\begin{align}\label{zeta}
\Psi(W, f_s)=\int_{U_{\SO_{3}}\backslash\SO_{3}}\int_{R^{l,1}}W(rw^{l,n} h) f_s(h,1)  \,dr\,dh.
\end{align}
It admits meromorphic continuation to a rational function in $q^{-s}$.

Next consider the intertwining operator
\begin{align*}
M(\tau,s):V(\tau,s)\rightarrow V(\tau^{-1},1-s)
\end{align*}
given by the meromorphic continuation of the integral
\begin{align*}
M(\tau,s)f_s(h,a) = \int_{U_{\SO_3}} f_s(w_1 u h, -a^{-1}) d u,\qquad w_1=\left(\begin{smallmatrix}
&  & 1 \\& -1 &  \\1 &  & \end{smallmatrix}\right).
\end{align*}
The measure $du$ is the additive measure $dx$ of $F$, where we identify $u\in U_{\SO_3}$ with $F$ via $u\mapsto u_{1,2}$.
The normalized intertwining operator $M^*(\tau,s)=C(s,\tau,\psi)M(\tau,s)$ is defined by the functional equation
\begin{align}\label{shahidigammafactor}
&\int_{U_{\SO_3}} f_s(w_1u,1) \psi^{-1}(u_{1,2}) \,du
=C(s,\tau,\psi)\int_{U_{\SO_3}} M(\tau,s)f_s(w_1u,1) \psi^{-1}(u_{1,2}) \,du.
\end{align}
Note that we omitted the matrix $d_1=-1$ appearing on both sides of this equation in \cite[(3.5)]{Kaplan2015}, because $\tau(-1)=\tau^{-1}(-1)$.
The constant $C(s,\tau,\psi)$ is essentially Shahidi's $\gamma$-factor $\gamma(2s-1, \tau, S^2, \psi)$ defined in \cite{Sh3}, up to a factor of the form $Bq^{As}$ where $A$ and $B$ are constants depending only on $\tau$, $\psi$ and $F$. For our purpose here we need to find the precise value of $C(s,\tau,\psi)$ and we have the following proposition, proved in \S~\ref{normalization parameters} below.
\begin{prop}\label{proposition:C factor}
Let $\gamma^{\mathrm{Tate}}(s,\tau^2,\psi)$ be the $\gamma$-factor of Tate \cite{Tate} (see \S~\ref{normalization parameters}). We have
\begin{align}\label{eq:C factor}
C(s,\tau,\psi)&=\tau^4(2)|2|^{4s}\tau^{-1}(\gamma)|\gamma|^{-s-1}\gamma^{\mathrm{Tate}}(2s-1,\tau^2,\psi).
\end{align}
\end{prop}

The integral $\Psi^*(W, f_s)=\Psi(W, M^*(\tau,s)f_s)$ is absolutely convergent in $\Re(s)\ll0$, and the functional equation is defined by
\begin{align}\label{gamma def}
\gamma(s,\pi\times\tau,\psi)\Psi(W, f_s)=\pi(-I_{2l})\tau(-1)^{l}\left(\tau^2(2)|2|^{2s-1}\tau^{-2}(\gamma)|\gamma|^{-2s+1}\right)\Psi^*(W, f_s).
\end{align}
Since the definition of $\SO_{3}$ here and the choice of vector $e$ are different from \cite{Kaplan2015}, the normalization
factor appearing on the right hand side of \eqref{gamma def} is different. We compute this factor, i.e., prove \eqref{gamma def}, in \S~\ref{normalization parameters}.

\begin{rmk}
In the split case in \cite{Kaplan2015}, the parameter $\gamma$ was chosen such that $2\gamma=\rho$ was a square, because the same parameter was used for
the embedding of $\SO_{2l}$ in $\SO_{2n+1}$. The group $\SO_{2l}$ was embedded in $\SO_{2n+1}$ in the stabilizer of a character of a unipotent subgroup of $\SO_{2n+1}$. That character depended on $\gamma$, and its stabilizer contained either the split or the quasi-split and nonsplit $\SO_{2l}$, depending on $\rho$. Also note that $\SO_{2n+1}$ here is isomorphic to the special orthogonal group defined with respect to $J_{2n+1}$ in \cite{Kaplan2015} (for any $\gamma\in F^*$). The factor $\tau^{-2}(\gamma)|\gamma|^{-2s+1}$ in
\eqref{gamma def} was denoted $c(s,l,\tau,\gamma)$ in \textit{loc. cit.}
\end{rmk}

\subsection*{Acknowledgements}
We would like to thank Gordan Savin for helpful conversations. Support to Adrian was provided by a grant from the Simons Foundation \#422638 and by a PSC-CUNY award, jointly funded by the
Professional Staff Congress and The City University of New York. Kaplan was
supported by the Israel Science Foundation, grant number 421/17.

\section{The simple supercuspidal representations of $\SO_{2l}$}\label{simple supercuspidals}

In this section, we recall the construction of the simple supercuspidal representations of $\SO_{2l}$.
Let $\Delta_{\SO_{2l}}=\{\epsilon_1 - \epsilon_2, \ldots, \epsilon_{l-1} - \epsilon_l, \epsilon_{l-1} + \epsilon_l\}$
denote the set of simple roots of $\SO_{2l}$, determined by our choice of the Borel subgroup $B_{\SO_{2l}}$.
Let $X^*(T_{\SO_{2l}})$ denote the character lattice of $T_{\SO_{2l}}$ and $T_0$ be the maximal compact subgroup of $T_{\SO_{2l}}$. Set
\begin{align*}
T_1 = \langle t \in T_0 : \lambda(t) \in 1 + \mathfrak{p} \ \forall \lambda \in X^*(T_{\SO_{2l}}) \rangle.
\end{align*}

We also have the set of affine roots $\Psi$, and we denote the subset of
simple affine roots by $\Pi$ and positive affine roots by $\Psi^+$. For $\psi \in \Psi$, $U_{\psi}$ is the
associated affine root group in $\SO_{2l}$. With our identifications,
\begin{align*}
\Pi = \{\epsilon_1 - \epsilon_2, \epsilon_2 - \epsilon_3,\ldots, \epsilon_{r-1} - \epsilon_r, \epsilon_{r-1} + \epsilon_r, 1 - \epsilon_1 - \epsilon_2 \}.
\end{align*}
Also define
\begin{align*}
I = \langle T_0, U_{\psi} : \psi \in \Psi^+ \rangle,\qquad I^+ = \langle T_1, U_{\psi} : \psi \in \Psi^+ \rangle.
\end{align*}

Let $\psi$ be a character of $F$ of level $1$.
According to \cite{GR10, RY14}, the \emph{affine generic characters} of $I^+$ take the form
\begin{align*}
\chi_{\underline{a}}(h) = \psi(a_1h_{1,2} + a_2 h_{2,3} + \cdots + a_{l-1} h_{l-1,l} + a_n h_{l-1,l+1} + a_{l+1} \frac{h_{2l-1,1}}{\varpi}), \ \ h \in I^+,
\end{align*}
where $\underline{a}=(a_1, a_2, ..., a_{l+1})\in (\mathfrak{o}^{\times})^{l+1}$, and because the level of $\psi$ is $1$,
we can further assume $a_i\in\kappa^{\times}$ for each $i$.
A complete set of representatives of $T_0$-orbits of affine generic characters of $I^+$ are given by the tuples
$(1, 1, \ldots, 1, 1,\alpha, t)$, where $\alpha$ varies over $\kappa^{\times} / (\kappa^{\times})^2$ and $t \in \kappa^{\times}$. Instead of viewing the affine generic characters of $I^+$ as parameterized by $\alpha \in \kappa^* / (\kappa^*)^2$ and $t \in \kappa^*$, we will set $t = 1$ and let the affine generic characters be parametrized by $\alpha \in \kappa^* / (\kappa^*)^2$ and the various choices of uniformizer $\varpi$ in $F$.

Let $x$ be the barycenter of the fundamental alcove and, for simplicity of notation, set $\chi = \chi_{\alpha} = \chi_{\underline{a}}$, noting that (up to the choice of a uniformizer $\varpi$), this character depends only on $\alpha$.
Simple supercuspidal representations are constructed using induction from compact subgroups. We describe this construction by explicating \cite[\S2]{RY14} for $\SO_{2l}$.

Put
\begin{align*}
g_{\chi} =
\begin{pmatrix}
 & & & & & -\varpi^{-1} \\
& I_{l-2} & 0 & 0 & 0\\
& 0 & 0  & \alpha^{-1} &0 &\\
& 0 & \alpha & 0 &0 \\
& 0 & 0 & 0 & I_{l-2}\\
-\varpi & & & & &
\end{pmatrix}\in\SO_{2l},
\end{align*}
and note that $g_{\chi}$ stabilizes $\chi$ since
\begin{align*}
\chi(g_{\chi} h g_{\chi}^{-1}) = \psi(-\varpi^{-1} h_{2l,2} + h_{23} + \cdots + h_{l-1, l} + \alpha h_{l-1,l+1} - h_{2l-1,2l})
\end{align*}
and the form defining $\SO_{2l}$ implies $h_{2l,2} = -h_{2l-1,1}$ and $h_{2l-1,2l} = -h_{12}$.

Let $H_{x, \chi} = Z \langle g_{\chi} \rangle I^+$
($H_{x,\chi}$ was defined in greater generality in \cite[\S2]{RY14}). Let $\omega$ be a character of $Z$, thereby extending $\chi$ from $I^+$ to $ZI^+$.  We may extend $\chi$ to a character $\chi_{\alpha}^{\omega}$ of $H_{x, \chi}$ by setting $\chi_{\alpha}^{\omega}(g_{\chi}) = \pm 1$, since $g_{\chi}^2 =1$.  After such an extension, which we denote again by $\chi$, we have that $\pi = \pi_{\alpha}^{\omega} =  Ind_{H_{x,\chi}}^G \chi$ is a simple supercuspidal representation (see \cite[\S2]{RY14}).

Define the character
\begin{align*}
\psi_{\alpha}(u) = \psi(\sum_{i=1}^{l-2} u_{i,i+1} + u_{l-1,l} + \alpha u_{l-1,l+1}),\qquad u\in \SO_{2l}.
\end{align*}
The representation $\pi$ is $\psi_{\alpha}$-generic.
For the purpose of constructing the integral,
put
\begin{align*}
\iota=\diag(I_{l-1},1/4,4,I_{l-1}),\qquad \alpha=-\tfrac\gamma4.
\end{align*}
By the definition of affine generic characters, $|\gamma / 4| = 1$. Then we have the isomorphic representation $\pi^{\iota}$, defined on the space of $\pi$ by
\begin{align*}
\pi^{\iota}(g)=\pi({}^{\iota}g)=\pi(\iota^{-1}g\iota).
\end{align*}
For $W\in W(\pi,\psi_{-\gamma/4}^{-1})$, define $W^{\iota}(g)=W({}^{\iota}g)$.
The map $W\mapsto W^{\iota}$ is an isomorphism
$W(\pi,\psi_{-\gamma/4}^{-1})\cong W(\pi^{\iota},\psi^{-1})$, where for the model $W(\pi^{\iota},\psi^{-1})$, $\psi$ is defined by
\eqref{def:character of U for integral}. Since $W(\pi^{\iota},\psi^{-1})$ is also a Whittaker model for $\pi$, definition \eqref{gamma def} implies
\begin{align*}
\gamma(s,\pi\times\tau,\psi)=\gamma(s,\pi^{\iota}\times\tau,\psi).
\end{align*}

\section{The computation of $\gamma(s,\pi\times\tau,\psi)$}\label{easyside}
Throughout this section and \S~\ref{The Langlands parameter},
$\psi$ is taken to be of level $1$ and $\tau$ is tamely ramified.
In this section, we compute $\gamma(s,\pi\times\tau,\psi)$ using a specific choice of data. Recall
$\pi = \mathrm{Ind}_{Z \langle g_{\chi} \rangle I^+}^{\SO_{2l}} \chi$. Let
$I_{\SO_3}^+$ be the pro-unipotent part of the standard Iwahori subgroup of $\SO_3$.
For the Whittaker function, consider $W=(w^{l,1})^{-1}\cdot W_0$ where
$W_0\in \mathcal{W}(\pi,\psi_{-\gamma/4}^{-1})$ is given by
\begin{align*}
W_0(g) =
\begin{cases}
\psi_{-\gamma/4}^{-1}(u) \chi(g_{\chi}^i) \omega(z) \chi(y) & g = u g_{\chi}^i z  y,\quad u \in U_{\SO_{2l}}, z\in Z, i=0,1,y\in I^+ ,\\
0  &  \text{otherwise.}
\end{cases}
\end{align*}
Then $W^{\iota}\in W(\pi^{\iota},\psi^{-1})$. Define $f_s$ by
\begin{align*}
 f_s(g,a)= \begin{cases}
|m|^{s} \tau(am)  & g = \diag(m,1,m^{-1})uy,\quad m\in F^*, u\in U_{\SO_3}, y\in I_{\SO_3}^+,\\
0 &  \text{otherwise.}
\end{cases}
\end{align*}

We can write a general element of the lower Borel subgroup $\overline{B}_{\SO_3}$ in the form
\begin{align*}
b=
\begin{pmatrix}
    a &  &  \\
     & 1 &  \\
    && a^{-1}
  \end{pmatrix}\begin{pmatrix}
    1 &  &  \\
    x & 1 &  \\
    -\tfrac\gamma4x^2 & -\tfrac\gamma2x& 1
  \end{pmatrix},\qquad a\in F^*, x\in F.
\end{align*}
Define the right invariant Haar measure $db$ on $\overline{B}_{\SO_3}$ by $db=|a|^{-1}d^*adx$.

Note that since $\SO_3$ is defined with respect to $J_{\gamma/2}$, if $b$ is such that $a=1$, then it belongs to $I_{\SO_3}^+$ if and only if $x\in\mathfrak{p}$.

\begin{lem}\label{support}
Assume $b$ as above.
\begin{enumerate}[leftmargin=*]
\item\label{it:support 1} If ${}^{\iota}(rw^{l,1}b(w^{l,1})^{-1})\in UZI^+$,
then $a \in 1 + \mathfrak{p}$, $|x|<1$, $r\in\mathfrak{p}^{l-2}$.
\item\label{it:support 2} If ${}^{\iota}(rw^{l,1}b(w^{l,1})^{-1})\in Ug_{\chi}ZI^+$, then for some $k\geq0$, we have
$|a|=q^{2k+1}$, $|x|=q^k$ and $\tfrac{\gamma}4x^2a^{-1}\in  \varpi \cdot (1 + \mathfrak{p})$, and also $r\in\mathfrak{p}^{l-2}$.
\end{enumerate}
\end{lem}
\begin{proof}
The image of $b$ in $\SO_{2l}$ is given by
\begin{align}\label{b in SO 2l}
\begin{pmatrix}
I_{l-2} & & & & &\\
& a & &  & \\
& \tfrac14 x & 1 &  & &\\
& \gamma x &  & 1 & \\
& -\tfrac\gamma4x^2a^{-1} & -\gamma xa^{-1} & -\tfrac14xa^{-1} & a^{-1}\\
& & & & &  I_{l-2}
\end{pmatrix},
\end{align}
and
\begin{align*}
{}^{\iota}(rw^{l,1}b(w^{l,1})^{-1})=
\begin{pmatrix}
a &  & & &\\
ar & I_{l-2} & & &\\
x & & 1& & \\
\tfrac\gamma4x & & & 1 & \\
& & &  & I_{l-2} \\
-\tfrac\gamma4x^2a^{-1} & & -\tfrac\gamma4xa^{-1} &  -xa^{-1} & r' & a^{-1}
\end{pmatrix}.
\end{align*}
Both assertions follow from this; note first that ${}^{\iota}(rw^{l,1}b(w^{l,1})^{-1}) \notin U \cdot z \cdot I^+$, where $z$ denotes the nontrivial central element in $\SO_{2l}$.  Also note that we are using $|\tfrac\gamma4|=1$, and for the second part of the Lemma,
consider the entries of the last row of an element in $U g_{\chi} Z I^+$.
\end{proof}

\begin{cor}\label{cor:psi W f}
We have
$\Psi(W^{\iota},f_s)= \mathrm{vol}^{\times}(1 + \mathfrak{p})\mathrm{vol}(\mathfrak{p})^{l-1}$.
\end{cor}
\begin{proof}
We may write the $dh$-integral of $\Psi(W^{\iota},f_s)$ over $\overline{B}_{\SO_3}$, then
\begin{align*}
\Psi(W^{\iota},f_s) = \int_{\overline{B}_{\SO_3}} (\int_{R^{l,1}}
W_0({}^{\iota}(rw^{l,1} b (w^{l,1})^{-1}))f_s(b,1) \,dr\,db.
\end{align*}
For $b$ as in Lemma \ref{support} (2), $f_s(b,1) = 0$. Hence
we are in part~\eqref{it:support 1} of Lemma~\ref{support}, and we conclude that the integrand vanishes unless
$a \in 1 + \mathfrak{p}$, $|x|<1$ and $r\in \mathfrak{p}^{l-2}$. In this case the integrand is identically $1$
(since $\tau$ is tamely ramified) and the result follows.
\end{proof}

We turn to compute $\Psi^*(W^{\iota}, f_s)$. We start with computing $M(\tau,s)f_s$ on the support of $W$, in the next two lemmas.
\begin{lem}\label{lem:intertwining operator when a is in 1+p}
Assume $\tau$ is quadratic, $a \in 1 + \mathfrak{p}$ and $|x|<1$. Then
\begin{align*}
M(\tau,s)f_s(b,1)=|\gamma|^s\tau(\gamma)(q-1)|2|^{1-2s}\frac{q^{1/2-2s}}{(1-q^{1-2s})}.
\end{align*}
\end{lem}
\begin{proof}
Since $\tau$ is tamely ramified, we can already assume $a=1$.
Write
\begin{align}\label{u}
u = \begin{pmatrix}
1 &  v & -\gamma^{-1} v^2\\
 & 1 & -2\gamma^{-1} v\\
 &  & 1
\end{pmatrix}
\end{align}
and we can already assume $v\ne0$ for the computation of $M(\tau,s)$ (the measure of a singleton is zero).
Then
\begin{align*}
w_1 u=
\begin{pmatrix}
1 & \gamma v^{-1} & -\gamma v^{-2}\\&1&-2v^{-1}\\&&1
\end{pmatrix}
\begin{pmatrix}
-\gamma v^{-2} \\2v^{-1}&1\\1&v&-\gamma^{-1}v^2
\end{pmatrix},
\end{align*}
and because $f_s$ is left-invariant under $U_{\SO_3}$,
\begin{align*}
M(\tau,s) f_s(b, 1)  &= \int_{F^{\times}} f_s (
\left(\begin{smallmatrix}
-\gamma  v^{-2} \\2v^{-1}&1\\1&v&-\gamma^{-1} v^2
\end{smallmatrix}\right)
\left(\begin{smallmatrix}
1 &  & \\
x & 1 & \\
-\tfrac\gamma4x^2 & -\tfrac\gamma2x & 1
\end{smallmatrix}\right), -1) \,dv
\\&= |\gamma|^s\tau(\gamma)\int_{F^{\times}}
\tau(v^{-2})|v|^{-2s}f_s(
\begin{pmatrix}
1\\2v^{-1}+x&1\\-\tfrac\gamma4(2v^{-1}+x)^2&-\tfrac\gamma2(2v^{-1}+x)&1
\end{pmatrix}, 1) \,dv.
\end{align*}
Since $f_s$ is supported in $B_{\SO_3}I_{\SO_3}^{+}$, considering the entries in the last row we deduce $2v^{-1} + x \in \mathfrak{p}$, and because
$|x|<1$, we find that
$2 v^{-1}\in \mathfrak{p}$. We obtain
\begin{align*}
|\gamma|^s\tau(\gamma)\int_{\{v\in F^{\times}: |2 v^{-1}| < 1 \} }
\tau(v^{-2})|v|^{-2s} \,dv.
\end{align*}
Using $dv=(q-1)q^{-1/2}|v|d^{\times}v$ and changing $v\mapsto 2 v$, we have
\begin{align*}
(q-1)q^{-1/2} \tau(2^{-2}) |2|^{1-2s}   \int_{\{v\in F^{\times}:|v|>1\}}\tau(v^{-2})|v|^{1-2s} \,d^{\times}v.
\end{align*}
Writing $v=\varpi^{-l}o$ with $|o|=1$,
\begin{align*}
\int_{\{v\in F^{\times}:|v|>1\}}\tau(v^{-2})|v|^{1-2s} \,d^{\times}v
=\sum_{l=1}^{\infty}q^{l(1-2s)}\tau(\varpi^{2l})\int_{\mathfrak{o}^{\times}}\tau^2(o)d^{\times}o.
\end{align*}
Now if $\tau$ is not quadratic, the $d^{\times}o$-integral vanishes and the result holds. Otherwise using
$\mathrm{vol}^{\times}(\frak{o}^{\times})=1$ and $\tau(\varpi^{2l})=1$, again we obtain the result.
\end{proof}

\begin{lem}\label{lem:operatoe on second coset}
Assume $\tau$ is quadratic and $|x|=q^k$ with $k\geq0$. Then
\begin{align*}
&M(\tau,s)f_s(b,1)=|\gamma|^s\tau(\gamma)|a|^{1-s}\tau^{-1}(a)|2|^{1-2s} q^{2k(s-1)} \mathrm{vol}(\mathfrak{p}).
\end{align*}
\end{lem}
\begin{proof}
As in the proof of Lemma~\ref{lem:intertwining operator when a is in 1+p} and with the same notation (but for any $a$),
\begin{align*}
&M(\tau,s) f_s(b, 1) \\&=|\gamma|^s\tau(\gamma)|a|^{1-s}\tau^{-1}(a)\int_{F^{\times}}
\tau(v^{-2})|v|^{-2s}f_s(
\begin{pmatrix}
1\\2v^{-1}+x&1\\-\tfrac\gamma4(2v^{-1}+x)^2&-\tfrac\gamma2(2v^{-1}+x)&1
\end{pmatrix}, 1) \,dv.
\end{align*}
Since $\tau$ is quadratic, $\tau(v^{-2})=1$. Changing $v\mapsto 2v$ we obtain
\begin{align*}
&|\gamma|^s\tau(\gamma)|a|^{1-s}\tau^{-1}(a)|2|^{1-2s}\int_{F^{\times}}
|v|^{-2s}f_s(
\begin{pmatrix}
1\\v^{-1}+x&1\\-\tfrac\gamma4(v^{-1}+x)^2&-\tfrac\gamma2(v^{-1}+x)&1
\end{pmatrix}, 1) \,dv.
\end{align*}
Again the integrand vanishes unless $v^{-1}+x\in \mathfrak{p}$, or $v^{-1}\in-x(1+ x^{-1}\mathfrak{p})$. Since
$|x|\geq1$, the additive group $x^{-1}\mathfrak{p}$ is contained in $\mathfrak{p}$. Thus
$|v|=q^{-k}$ and the $dv$-integral equals $q^{2k(s-1)} \mathrm{vol}(\mathfrak{p})$. The result follows.
\end{proof}

\begin{lem}\label{W on other coset}
Assume ${}^{\iota}(rw^{l,1}b(w^{l,1})^{-1})\in Ug_{\chi}ZI^+$. Then
\begin{align*}
\int_{R^{l,1}} W_0^{\iota}(r w^{l,1} b (w^{l,1})^{-1}) \,dr = \chi(g_{\chi})\mathrm{vol}(\mathfrak{p})^{l-2}.
\end{align*}
\end{lem}
\begin{proof}
Write the image of $b$ in $\SO_{2l}$ as in \eqref{b in SO 2l} and consider
\begin{align*}
u=\begin{pmatrix}
I_{l-2}\\
&1&-4ax^{-1}&-\tfrac1\gamma ax^{-1}&-\tfrac4\gamma a^2x^{-2}\\
&&1&&\tfrac1\gamma ax^{-1}\\
&&&1&4ax^{-1}\\
&&&&1\\
&&&&&I_{l-2}
\end{pmatrix}.
\end{align*}
Then
\begin{align*}
u b=\begin{pmatrix}
I_{l-2}\\&&&&-\tfrac4\gamma ax^{-2}\\&&&-\tfrac1{4\gamma} &\tfrac1\gamma x^{-1}\\&&-4\gamma &0&4x^{-1}\\
&-\tfrac\gamma4x^2a^{-1}&-\gamma xa^{-1}&-\tfrac14xa^{-1}&a^{-1}\\&&&&&I_{l-2}\end{pmatrix}
\end{align*}
and $g_{\chi}w^{l,1}ub(w^{l,1})^{-1}=\iota^{-1}y\iota^{-1}$ where
\begin{align*}
y=\begin{pmatrix}
\tfrac\gamma4x^2a^{-1}\varpi^{-1}&0&\tfrac\gamma4 xa^{-1}\varpi^{-1}&xa^{-1}\varpi^{-1}&0&-a^{-1}\varpi^{-1}\\
0&I_{l-2}&0&0&0&0\\
0&0&1&0&0&-\tfrac4\gamma x^{-1}\\0&0&0&1&0&-x^{-1}\\
0&0&0&0&I_{l-2}&0\\
0&0&0&0&0&\tfrac4\gamma ax^{-2}\varpi
\end{pmatrix}.
\end{align*}
By Lemma~\ref{support}~\eqref{it:support 2}, $|a|=q^{2k+1}$, $|x|=q^k$ for some $k\geq0$ and $\tfrac{\gamma}4x^2a^{-1}\in  \varpi \cdot (1 + \mathfrak{p})$, and because $|\tfrac4\gamma|=1$, we deduce $y\in I^+$. Also note that $g_{\chi}=g_{\chi}^{-1}$ and $g_{\chi}\iota^{-1}=\iota g_{\chi}$. Then
\begin{align*}
&b(w^{l,1})^{-1}=u^{-1}(w^{l,1})^{-1} g_{\chi} \iota^{-1}y\iota^{-1}=u^{-1}(w^{l,1})^{-1} \iota g_{\chi} y\iota^{-1},\\
&W_0^{\iota}(r w^{l,1} b (w^{l,1})^{-1}) = W_0^{\iota}(r w^{l,1}u^{-1}(w^{l,1})^{-1}\iota g_{\chi}y\iota^{-1}).
\end{align*}

Next we see that
\begin{align*}
r (w^{l,1} u^{-1} (w^{l,1})^{-1}) r^{-1} =
v_{r,u}=\begin{pmatrix}
   1&&4ax^{-1}&\tfrac1\gamma ax^{-1}&*&*\\&I_{l-2}&4ax^{-1}r&\tfrac1\gamma ax^{-1}r&*&*\\
&&1&&*&-\tfrac1\gamma ax^{-1}\\&&&1&*&-4ax^{-1}\\&&&&I_{l-2}&\\
&&&&&1\end{pmatrix}\in U_{\SO_{2l}},
\end{align*}
where the remaining coordinates are determined by $a,x,r$ and the form defining $\SO_{2l}$.
Since
\begin{align*}
\psi(v_{r,u})=\psi(\tfrac14(4ax^{-1}r)-\gamma(\tfrac1\gamma ax^{-1}r))=1,
\end{align*}
$W_0^{\iota}$ is right-invariant under $\iota I^+\iota^{-1}$ and $\iota$ commutes with $r$,
\begin{align*}
&W_0^{\iota}(r w^{l,1} b (w^{l,1})^{-1}) = W_0^{\iota}(v_{r,u}r\iota g_{\chi}\iota^{-1})=
W_0^{\iota}(r\iota g_{\chi}\iota^{-1})=W_0(rg_{\chi}).
\end{align*}
Now, we have $W_0(r g_{\chi}) = W_0(g_{\chi} ({}^{g_{\chi}}r))$, ${}^{g_{\chi}}r\in I^+$ and $\chi({}^{g_{\chi}}r)=1$, hence $W_0(r g_{\chi}) = \chi(g_{\chi})$.
\end{proof}
\begin{cor}\label{cor:psi *}$\Psi(W^{\iota},M(\tau,s)f_s)$ equals
\begin{align*}
\Psi(W^{\iota},f_s)(q-1)|2|^{1-2s} |\gamma|^s \left(  \tau(\gamma)\frac{q^{1/2-2s}}{(1-q^{1-2s})}
+\chi(g_{\chi})\tau(\varpi) q^{-1/2-s}/(1-q^{-1})\right).
\end{align*}
\end{cor}
\begin{proof}
As in the proof of Corollary~\ref{cor:psi W f}, we write the $dh$-integral over $\overline{B}_{\SO_3}$ and using $db$. According to the support
of $W^{\iota}$ and by Lemma~\ref{support}, $\Psi(W^{\iota},M(\tau,s)f_s)$ is the sum of two integrals, each corresponding to one of the cases of the lemma. The first summand is
\begin{align*}
&\int_{a\in1+\mathfrak{p}}
\int_{x\in\mathfrak{p}}
\int_{r\in\mathfrak{p}^{l-2}}W_0^{\iota}(rw^{l,1}b(w^{l,1})^{-1})[M(\tau,s) f_s](b,1)\,dr\,da\,dx
\\&=  \mathrm{vol}^{\times}(1+\mathfrak{p})  \mathrm{vol}(\mathfrak{p})^{l-1} |\gamma|^s\tau(\gamma) (q-1)  |2|^{1-2s}      \frac{q^{1/2-2s}}{(1-q^{1-2s})},
\end{align*}
by Lemma~\ref{lem:intertwining operator when a is in 1+p}.

The second summand is an integral over $(a,x)$ such that
\begin{align*}
|a|=q^{2k+1},\qquad |x|=q^k, \qquad \tfrac{\gamma}4x^2a^{-1}\in  \varpi \cdot (1 + \mathfrak{p}),\qquad k\geq0.
\end{align*}
For each such elements, by Lemma~\ref{lem:operatoe on second coset} and Lemma~\ref{W on other coset} the integrand equals
\begin{align*}
|\gamma|^s\tau(\varpi)|2|^{1-2s} q^{1-s} \chi(g_{\chi})\mathrm{vol}(\mathfrak{p})^{l-1},
\end{align*}

where we also used $\tau^{-1}(a)=\tau(\gamma)\tau(\varpi)$ ($\tau$ is quadratic and tamely ramified). It remains to compute the measure $db$ in the integral, and
because $db=|a|^{-1}da\, dx$, $|xa^{-1}|=q^{-k-1}$ and
$\mathrm{vol}^{\times}(\tfrac{\gamma}4x^{2}\varpi^{-1}(1+\mathfrak{p}))=\mathrm{vol}^{\times}(1+\mathfrak{p})$, we have
\begin{align*}
q^{-1}\mathrm{vol}(\mathfrak{o}^{\times})\mathrm{vol}^{\times}(1+\mathfrak{p}))\sum_{k=0}^{\infty}q^{-k}
=(q-1)q^{-3/2}\mathrm{vol}^{\times}(1+\mathfrak{p})\frac1{1-q^{-1}}.
\end{align*}
Thus
\begin{align*}
\Psi(W^{\iota},M(\tau,s)f_s)=&\mathrm{vol}^{\times}(1+\mathfrak{p})\mathrm{vol}(\mathfrak{p})^{l-1}(q-1)|2|^{1-2s} |\gamma|^s
\\&\left(  \tau(\gamma)\frac{q^{1/2-2s}}{(1-q^{1-2s})}
+\chi(g_{\chi})\tau(\varpi) q^{-1/2-s}/(1-q^{-1})\right),
\end{align*}
since $\tau$ is quadratic.  The formula follows when we plug Corollary~\ref{cor:psi W f} into this identity.
\end{proof}

\begin{cor}\label{corollary:gamma for tamely ramified}
For any quadratic tamely ramified character $\tau$ of $F^*$,
\begin{align*}
\gamma(s,\pi\times\tau,\psi)=&
\pi(-I_{2l})\tau(-1)^{l}\tau(\gamma)(q-1)
\gamma^{\mathrm{Tate}}(2s-1,\tau^2,\psi)
\\&\times\left(  \tau(\gamma)\frac{q^{1/2-2s}}{(1-q^{1-2s})}
+\chi(g_{\chi})\tau(\varpi) q^{-1/2-s}/(1-q^{-1})\right).
\end{align*}
\end{cor}
\begin{proof}
Use \eqref{gamma def}, Corollary~\ref{cor:psi W f}, Corollary~\ref{cor:psi *}, \eqref{eq:C factor}, and note that $\tau^2=1$ and $|\gamma|=|4|$.
\end{proof}
Theorem~\ref{mainmain} now follows immediately from Corollary~\ref{corollary:gamma for tamely ramified}.

\section{The Langlands parameter}\label{The Langlands parameter}

In this section we discuss the Langlands parameter for $\pi$.  Recall that $\pi = \pi_{\alpha}^{\omega}$ is a simple supercuspidal representation of $\SO_{2l}$, corresponding to the character $\chi=\chi_{\alpha}^{\omega}$. Let $\varphi=\varphi_{\pi}$ be the Langlands parameter of $\pi$.

First assume $p \neq 2$. It is then expected that
\begin{align}\label{eq:Langlands parameter p not 2}
\varphi = \varphi_1 \oplus \varphi_2 \oplus \varphi_3,
\end{align}
where $\varphi_1$ and $\varphi_2$ are $1$-dimensional, and
$\varphi_3$ is $2l-2$ dimensional and corresponds, via the local Langlands correspondence, to a simple supercuspidal representation $\Pi'$ of $\GL_{2l-2}$.

Let $\tau_1,\tau_2$ be the quadratic characters of $F^*$ such that $\gamma(s,\pi\times\tau_i,\psi)$ has a pole at $s=1$, guaranteed by
Theorem \ref{mainmain}. We let $\tau_1$ be the unramified such character, and $\tau_2$ the tamely ramified one. Without loss of generality, $\varphi_1 = \tau_1$ and $\varphi_2 = \tau_2$.
Since $\det\varphi=1$, the central character of $\Pi'$ equals
$\tau_1\tau_2$. Let $\delta$ be the coefficient of $q^{1/2-s}$ in $\gamma(s, \Pi', \psi)$.
\begin{prop}\label{proposition:delta for p not 2}
$\delta=\pi(-I_{2l}) \chi(g_{\chi}) \epsilon(s, \tau_2, \psi)^{-1}$.
\end{prop}
\begin{proof}
By the local Langlands correspondence for general linear groups, $\delta$ is precisely the coefficient of $q^{1/2-s}$ in $\gamma(s, \varphi_1, \psi)$ (see \cite[\S~2.6]{AdrianKaplan2018}). By \eqref{eq:Langlands parameter p not 2}, the local Langlands correspondence also implies
\begin{align*}
\gamma(s,\pi,\psi)=\gamma(s, \Pi', \psi)\gamma(s, \tau_1, \psi)\gamma(s, \tau_2, \psi).
\end{align*}
Then by Corollary~\ref{corollary:gamma for tamely ramified} (with $\tau\equiv1$),
\begin{align}\label{eq:first computation of gamma Pi'}
\gamma(s, \Pi', \psi)=&
\pi(-I_{2l})(q-1)q^{2s-3/2} \frac{1-q^{1-2s}}{1-q^{2s-2}}\left(\frac{q^{1/2-2s}}{(1-q^{1-2s})}
+\chi(g_{\chi})q^{-1/2-s}/(1-q^{-1})\right)
\\&\times\gamma(s, \tau_1, \psi)^{-1}\gamma(s, \tau_2, \psi)^{-1}.\nonumber
\end{align}
Here we used $\gamma^{\mathrm{Tate}}(2s-1,1,\psi)=q^{2s-3/2} \frac{1-q^{1-2s}}{1-q^{2s-2}}$.

Write $\gamma =4\cdot u$ for some $u \in \mathfrak{o}^{\times}$, and then
$\tau_i(\gamma)=\tau_i(u)$, $i=1,2$. Because $\tau_1$ is unramified and $\tau_2$ is the unique nontrivial quadratic character of $\mathfrak{o}^{\times}$,
\begin{align*}
\tau_1(\varpi) = \chi(g_{\chi}),\qquad \tau_2(\varpi) = \chi(g_{\chi}) \tau_2(\gamma).
\end{align*}
Now by virtue of \cite[\S~23.4, \S~23.5]{BH06},
\begin{align*}
\gamma(s, \tau_1, \psi) = q^{s-1/2} \chi(g_{\chi})\frac{1-\chi(g_{\chi}) q^{-s}}{1 - \chi(g_{\chi}) q^{s-1}},\qquad
\gamma(s, \tau_2, \psi) = \epsilon(s, \tau_2, \psi).
\end{align*}
Plugging this into \eqref{eq:first computation of gamma Pi'} we obtain
\begin{align*}
\gamma(s, \Pi', \psi)=\pi(-I_{2l}) \chi(g_{\chi}) \epsilon(s, \tau_2, \psi)^{-1} q^{1/2-s},
\end{align*}
as claimed.
\end{proof}
\begin{rmk}
Because $\tau_2$ is tamely ramified and not unramified, the power of $q$ in $\epsilon(s, \tau_2, \psi)$ is zero.
\end{rmk}
The final ingredient we need is the restriction of $\varphi_3$ to the wild inertia subgroup. Unfortunately, as mentioned in the introduction, our method does not provide this information. We expect it to follow from an analogue of \cite{BHS17}, but note that only the simple supercuspidal case of \cite{BHS17} is required here. Such an analogue would both prove that $\varphi$ does indeed contain a $2l-2$ dimensional summand which corresponds to a simple supercuspidal representation (we have assumed this, as it is expected), and also determine the restriction of $\varphi_3$ to the wild inertia subgroup. Then by \cite{AL14}, the central character of $\Pi'$, the parameter $\delta$ obtained by Proposition~\ref{proposition:delta for p not 2}, and the restriction of $\varphi_3$ to the wild inertia subgroup, completely determine $\varphi_3$. Since we already described $\varphi_1$ and $\varphi_2$, this completely determines the Langlands parameter $\varphi$ of $\pi$.

Now consider the case $p = 2$. Then we expect to have a decomposition
\begin{align}\label{eq:Langlands parameter p = 2}
\varphi = \varphi_1 \oplus \varphi_2,
\end{align}
where $\varphi_1$ is $1$-dimensional, and $\varphi_2$ is of dimension $2l-1$ and corresponds to a simple supercuspidal representation $\Pi'$ of $\GL_{2l-1}$. In this case by Theorem \ref{mainmain}, there is a unique quadratic character $\tau$ of $F^*$ such that $\gamma(s, \pi\times\tau, \psi)$ a pole at $s=1$. Thus $\varphi_1=\tau$. The central character of $\Pi'$ is automatically trivial since $-1 \in I^+$ ($p = 2$). For the same reason $\pi(-I_{2l}) = 1$. Then the computation in the proof of Propoition~\ref{proposition:delta for p not 2} (but without $\gamma(s, \tau_2, \psi)$) implies
\begin{align*}
\delta=\chi(g_{\chi}),
\end{align*}
where $\delta$ is the coefficient of $q^{1/2-s}$ in $\gamma(s,\Pi',\psi)$. When $F=\Q_2$ this completely determines the Langlands parameter $\varphi$, because in the parameterization of simple supercuspidal representations of general linear groups, the uniformizer may be chosen modulo $1 + \mathfrak{p}$. For more general $2$-adic fields, we still need to find the restriction of $\varphi_2$ to the wild inertia subgroup, which as mentioned above is expected to be obtained from an analogue of \cite{BHS17}. Note that while the results of \textit{loc. cit.} were obtained under the assumption $p\ne2$, at least for the class of simple supercuspidal representations an extension of their results to $p=2$ seems possible.

\section{The normalization parameters of $\gamma(s,\pi\times\tau,\psi)$}\label{normalization parameters}

In this section we prove Proposition~\ref{proposition:C factor}, i.e., compute $C(s,\tau,\psi)$, and determine the normalization factor of
\eqref{gamma def}. We start with some preliminaries.

Let $\mathcal{S}(F^r)$ be the space of Schwartz--Bruhat functions on the row space $F^r$ and let $(e_1,\ldots,e_r)$ be the standard basis of $F^r$. Define the Fourier transform $\widehat{\Phi}\in\mathcal{S}(F^{r})$ by
\begin{align*}
\widehat{\Phi}(y)=\int_{F^{r}}\Phi(z)\psi(z({}^t{y}))dz
\end{align*}
(here $y$ and $z$ are rows).

We recall the definition of Tate's $\gamma$-factor $\gamma^{\mathrm{Tate}}(s,\eta,\psi)$, for a quasi-character $\eta$ of $F^*$
\cite{Tate}. For $\Phi\in\mathcal{S}(F)$, consider the zeta integral
\begin{align*}
Z(\Phi,s,\eta)=\int_{F^*}\Phi(x)\eta(x)|x|^s\, d^*x,
\end{align*}
which is absolutely convergent for $\Re(s)\gg 0$ and admits meromorphic continuation to a function in $q^{-s}$.
The $\gamma$-factor is then defined via the functional equation
\begin{align}\label{eq:GJ even orthogonal non split minimal cases fun equation}
\gamma^{\mathrm{Tate}}(s,\eta,\psi)Z(\Phi,s,\eta)=Z(\widehat{\Phi},1-s,\eta^{-1}).
\end{align}

We calculate $C(s,\tau,\psi)$ by choosing a special section in $V(\tau,s)$, for which we can succinctly compute
its image under $M(\tau,s)$, then compare both sides of \eqref{shahidigammafactor}. We argue by adapting parts of the arguments from  \cite[\S~6.1]{Kaplan2015}. To construct the section we use an isomorphism $\imath:Z_{\GL_2}\backslash\GL_2\rightarrow \SO_3$, where
$Z_{\GL_2}$ is the center of $\GL_2$.

To define $\imath$, it is useful to consider the complex Lie algebras, in order to identify the images of unipotent elements. Let
\begin{align*}
A=\left(\begin{smallmatrix}0&1\\0&0\end{smallmatrix}\right),\quad
B=\left(\begin{smallmatrix}0&0\\1&0\end{smallmatrix}\right),\quad
C=\left(\begin{smallmatrix}1&0\\0&-1\end{smallmatrix}\right),\quad
D=\left(\begin{smallmatrix}1&0\\0&1\end{smallmatrix}\right)
\end{align*}
be a basis for the Lie algebra $\mathfrak{gl}_2$ of $\GL_2$ over $\C$. The center $\mathfrak{z}_{\mathfrak{gl}_2}$ is spanned by $D$.
Then
\begin{align*}
[A,B]=C,\qquad [A,C]=-2A,\qquad [B,C]=2B.
\end{align*}
Also let
\begin{align*}
X=\left(\begin{smallmatrix}0&0&0\\1&0&0\\0&-\gamma/2&0\end{smallmatrix}\right),\quad
Y=\left(\begin{smallmatrix}0&\gamma/2&0\\0&0&-1\\0&0&0\end{smallmatrix}\right),\quad
Z=\left(\begin{smallmatrix}1&0&0\\0&0&0\\0&0&-1\end{smallmatrix}\right)
\end{align*}
be a basis for the Lie algebra $\mathfrak{so}_3$ of $\SO_3$ over $\C$,
\begin{align*}
[X,Y]=-\tfrac\gamma2Z,\qquad [X,Z]=X,\qquad [Y,Z]=-Y.
\end{align*}
Hence the following defines an isomorphism of Lie algebras
$d\imath_0:\mathfrak{z}_{\mathfrak{gl}_2}\backslash\mathfrak{gl}_2\rightarrow\mathfrak{so}_3$:
\begin{align*}
d\imath_0(A)=X, \quad d\imath_0(B)=\tfrac4\gamma Y, \quad d\imath_0(C)=-2Z, \quad d\imath_0(D)=0.
\end{align*}
In particular
\begin{align*}
&\imath_0\left(\begin{smallmatrix}1&u\\&1\end{smallmatrix}\right)=
\left(\begin{smallmatrix}1&&\\u&1&\\-\tfrac\gamma4u^2&-\tfrac\gamma2u&1\end{smallmatrix}\right),\qquad
\imath_0\left(\begin{smallmatrix}1&\\u&1\end{smallmatrix}\right)=\left(\begin{smallmatrix}1&2u&-\tfrac4\gamma u^2\\&1&-\tfrac4\gamma u\\&&1\end{smallmatrix}\right).
\end{align*}
It follows that $\imath$ is determined by
\begin{align*}
&\imath\left(\begin{smallmatrix}a&\\&b\end{smallmatrix}\right)=\left(\begin{smallmatrix}a^{-1}b&&\\&1&\\&&ab^{-1}\end{smallmatrix}\right),\qquad
\imath\left(\begin{smallmatrix}1&\\u&1\end{smallmatrix}\right)=\left(\begin{smallmatrix}1&2u&-\tfrac4\gamma u^2\\&1&-\tfrac4\gamma u\\&&1\end{smallmatrix}\right),\qquad
\imath\left(\begin{smallmatrix}&\tfrac4\gamma\\1&\end{smallmatrix}\right)=\left(\begin{smallmatrix}&&1\\&-1&\\1&&\end{smallmatrix}\right)=w_1.
\end{align*}
If $h\in \SO_{3}$, let $\imath^{-1}(h)$ be an arbitrary pre-image of $h$ in $\GL_2$, under $\imath$.

\begin{proof}[Proof of Proposition~\ref{proposition:C factor}]
For $\Phi\in\mathcal{S}(F^2)$, define $f_{\Phi,\tau,s}\in V(\tau,s)$ by
\begin{align*}
f_{\Phi,\tau,s}(h,a)=\int_{Z_{\GL_2}}\Phi(e_1z\imath^{-1}(g))\tau(\det{z\imath^{-1}(g)})|\det{z\imath^{-1}(g)}|^s\,dz.
\end{align*}

Recall that $M(\tau,s)f_{\Phi,\tau,s}(h,a)=\int_{U_{\SO_3}}f_{\Phi,\tau,s}(w_1uh,-a^{-1})du$. Since
\begin{align*}
w_1u=
\left(\begin{smallmatrix}&&1\\&-1&\\1&&\end{smallmatrix}\right)
\left(\begin{smallmatrix}1&u&-\tfrac1\gamma u^2\\&1&-\tfrac2\gamma u\\&&1\end{smallmatrix}\right)
=\imath(\left(\begin{smallmatrix}&\tfrac4\gamma\\1&\end{smallmatrix}\right)
\left(\begin{smallmatrix}1&\\\tfrac12u&1\end{smallmatrix}\right))
=\iota\left(\begin{smallmatrix}\tfrac2\gamma u&\tfrac4\gamma\\1&0\end{smallmatrix}\right),
\end{align*}
\begin{align*}
M(\tau,s)f_{\Phi,\tau,s}(I_3,1)=&\tau(\tfrac4\gamma)|\tfrac4\gamma|^s\int_F\int_{F^*}\Phi(\tfrac2\gamma zu, \tfrac4\gamma z)\tau^2(z)|z|^{2s}d^*zdu.
\end{align*}
Changing variables $z\mapsto\tfrac\gamma4z$ and $u\mapsto 2z^{-1}u$, we have
\begin{align*}
M(\tau,s)f_{\Phi,\tau,s}(I_3,1)=\tau(\tfrac\gamma4)|\tfrac\gamma4|^s|2|Z(\Phi_1,2s-1,\tau^2),\qquad \Phi_1(z)=\int_F\Phi(u,z)du.
\end{align*}
This formal step is justified for $\Re(s)\gg0$ by Fubini's Theorem.
According to \eqref{eq:GJ even orthogonal non split minimal cases fun equation}, when we multiply $Z(\Phi_1,2s-1,\tau^2)$ by $\gamma^{\mathrm{Tate}}(2s-1,\tau^2,\psi)$ we get $Z(\widehat{\Phi_1},2-2s,\tau^{-2})$ (as meromorphic continuations) and because
\begin{align*}
\widehat{\Phi_1}(z)=\int_F\int_F\Phi(u,y)\psi(yz)\,du\,dy=\widehat{\Phi}(0,z),
\end{align*}
we have
\begin{align}\label{eq:minimal cases even orthogonal non split first formula for intertwining at identity}
&\gamma^{\mathrm{Tate}}(2s-1,\tau^2,\psi)M(\tau,s)f_{\Phi,\tau,s}(h,1)
\\&=\tau(\tfrac\gamma4)|\tfrac\gamma4|^s|2|\int_{Z_{\GL_2}}(\imath^{-1}(h)\cdot \Phi)^{\widehat{}}(e_2z)\tau^{-1}(\det z)\tau(\det\imath^{-1}(h))|\det{z}|^{1-s}|\det{\imath^{-1}(h)}|^{s}\,dz.\nonumber
\end{align}

Now we compute $C(s,\tau,\psi)$ by substituting $f_{\Phi,\tau,s}$ for $f_s$ in \eqref{shahidigammafactor}, which becomes
\begin{align}\label{eq:Shahidi func equation even orthogonal n=1 with substitution}
&\int_{U_{\SO_3}} f_{\Phi,\tau,s}(w_1u,1) \psi^{-1}(u_{1,2}) \,du
=C(s,\tau,\psi)\int_{U_{\SO_3}} M(\tau,s)f_{\Phi,\tau,s}(w_1u,1) \psi^{-1}(u_{1,2}) \,du.
\end{align}
Changing variables as above, but now also paying attention to $\psi^{-1}$, the left hand side is
\begin{align}\label{eq:minimal cases even orthogonal non split Shahidi functional equation left hand side}
\tau(-1)\tau(\tfrac4\gamma)|\tfrac4\gamma|^{s}|2|\int_{F^*}\left(\int_F\Phi(u,z)\psi^{-1}(2z^{-1}u)\,du\right)\tau^2(z)|z|^{2s-1}\,d^*z.
\end{align}
For the right hand side, we use \eqref{eq:minimal cases even orthogonal non split first formula for intertwining at identity}
with $h=w_1u$ and obtain
\begin{align*}
&\gamma^{\mathrm{Tate}}(2s-1,\tau^2,\psi)^{-1}\tau(-1)|2|\int_F\int_{Z_{\GL_2}}(\imath^{-1}(w_1u)
\cdot \Phi)^{\widehat{}}(e_2z)\tau^{-1}(\det z)|\det{z}|^{1-s}\,dz\,\psi^{-1}(u)\,du.
\end{align*}
Using $(g\cdot \Phi)^{\widehat{}}(x,y)=|\det g|^{-1}\widehat{\Phi}((x,y)({}^tg^{-1}))$ and changing $u\mapsto2z^{-1}u$, this equals
\begin{align}\label{eq:minimal cases even orthogonal non split helper 2}
&\gamma^{\mathrm{Tate}}(2s-1,\tau^2,\psi)^{-1}
\tau(-1)|\gamma|\int_{F^*}\left(\int_F
\widehat{\Phi}(z,u)\psi^{-1}(-2z^{-1}u)\,du\right)\tau^{-2}(z)|z|^{1-2s}\,d^*z.
\end{align}
Observe that for a fixed $z$, by partial Fourier inversion,
\begin{align*}
\int_F
\widehat{\Phi}(z,u)\psi(2z^{-1}u)\,du
&=\int_F\int_F\Phi(x,y)\psi(xz)\left(\int_F\psi((y+2z^{-1})u)\,du\right)\,dx\,dy\\&
=\int_F\Phi(x,-2z^{-1})\psi(xz)\,dx.
\end{align*}
Hence \eqref{eq:minimal cases even orthogonal non split helper 2} becomes
\begin{align}\label{eq:minimal cases even orthogonal non split Shahidi functional equation right hand side}
&\gamma^{\mathrm{Tate}}(2s-1,\tau^2,\psi)^{-1}
\tau^{-2}(2)|2|^{1-2s}\tau(-1)|\gamma|\int_{F^*}\left(\int_F\Phi(u,z)\psi^{-1}(2z^{-1}u)\,du\right)\tau^{2}(z)|z|^{2s-1}d^*z.
\end{align}
Dividing \eqref{eq:minimal cases even orthogonal non split Shahidi functional equation left hand side} by \eqref{eq:minimal cases even orthogonal non split Shahidi functional equation right hand side}, we conclude
\begin{align*}
C(s,\tau,\psi)&=\tau^4(2)|2|^{4s}\tau^{-1}(\gamma)|\gamma|^{-s-1}\gamma^{\mathrm{Tate}}(2s-1,\tau^2,\psi).
\end{align*}
This completes the proof of the proposition.
\end{proof}

To find the normalization factor appearing in \eqref{gamma def} we must follow the computations from \cite{Kaplan2013a,Kaplan2015}.
This factor is extracted from the multiplicativity properties \cite[(6.1), (6.2)]{Kaplan2015} and from the minimal case of $\SO_2\times\GL_1$, but since here we only consider split $\SO_{2l}$, the multiplicativity properties are sufficient.

Let $Q_r=M_r\ltimes U_r$ be the standard maximal parabolic subgroup of $\SO_{2l}$ whose Levi part $M_r=\GL_r\times\SO_{2(l-r)}$ if $r<l$, and $\{\diag(b,b^*):b\in\GL_l\}$ for $r=l$.
Let $P_{(n_1,n_2)}$ be a parabolic subgroup of $\GL_{n}$, $n=n_1+n_2$, containing the subgroup of upper triangular invertible matrices, whose Levi part is isomorphic to $\GL_{n_1}\times\GL_{n_2}$. We could in theory work with
$n=1$, but since the multiplicativity properties for the case $\pi=\Ind_{Q_r}^{\SO_{2l}}(\sigma\otimes\pi')$ with $r<l$ and $l>n$ were
obtained using the case $r<l<n$ and the multiplicativity for $\tau=\Ind_{P_{(n_1,n_2)}}^{\GL_{n}}(\tau_1\otimes\tau_2)$,
we actually need to consider the general $\SO_{2l}\times\GL_n$ construction. In this case for $l\leq n$, $\SO_{2l}$ is embedded in $\SO_{2n+1}$ which is defined with respect to $J_{2n+1}$, exactly as in \cite{Kaplan2015}; but for $l>n$, $\SO_{2n+1}$ is now defined using $J_{2n+1,\gamma}$.

The functional equation for all $l$ and $n$ takes following form.
Define the factor $c(s,l,\tau,\gamma)=\tau^{-2}(\gamma)|\gamma|^{n(-2s+1)}$ if $l>n$, otherwise $c(s,l,\tau,\gamma)=1$ (as in \cite{Kaplan2013a,Kaplan2015}).
Then we claim
\begin{align}\label{gamma def general n}
\gamma(s,\pi\times\tau,\psi)\Psi(W, f_s)=\pi(-I_{2l})^n\tau(-1)^{l}\left(\tau^2(2)|2|^{n(2s-1)}c(s,l,\tau,\gamma)\right)\Psi^*(W, f_s).
\end{align}
Here $\Psi(W, f_s)$ and $\Psi^*(W, f_s)$ are the $\SO_{2l}\times\GL_n$ integrals, described in \S~\ref{notation} for $n=1$ and in
\cite{Kaplan2015} for all $n$. Specializing \eqref{gamma def general n} to $n=1$, we obtain \eqref{gamma def}.
\begin{rmk}
The factor $\tau^2(2)|2|^{n(2s-1)}$ in \eqref{gamma def general n} is different from the corresponding one in \cite[p.~408]{Kaplan2015} ($|2\gamma|^{n(s-1/2)}\tau(2\gamma)$) because the embedding is different, see \S~\ref{notation}.
\end{rmk}

Inspecting \cite{Kaplan2013a,Kaplan2015}, the only multiplicativity property for $\gamma(s,\pi\times\tau,\psi)$ which is affected by the difference in the definition of $\SO_{2n+1}$ and choice of embedding (the vector $e$, see \S~\ref{notation}) here is the one for $r=l>n$, which was proved in \cite[\S~5.4]{Kaplan2013a}. This property is replaced by the following result, which implies \eqref{gamma def general n} (see \cite[\S~6]{Kaplan2015}).
\begin{prop}

Assume $\pi=\Ind_{\overline{Q}_l}^{\SO_{2l}}(\sigma)$,
where $\overline{Q}_l=M_l\ltimes\overline{U}_l$, and $\tau$ is an irreducible generic representation of $\GL_n$. Then
\begin{align}\label{eq:multiplicativity helper results first var even orthogonal}
\frac{\Psi^*(W, f_s)}{\Psi(W, f_s)}=\sigma(-I_{2l})^n\tau(-1)^l\tau^{-2}(2)|2|^{-2s+1}c(s,l,\tau,\gamma)^{-1}
\gamma(s,\sigma\times\tau,\psi)\gamma(s,\sigma^*\times\tau,\psi).
\end{align}
Here $\sigma^*$ is the representation on the space of $\sigma$ acting by $\sigma^*(b)=\sigma(b^*)$, and
the $\gamma$-factors are the Rankin--Selerg $\GL_l\times \GL_n$ $\gamma$-factors of \cite{JPSS}.
\end{prop}
\begin{proof}
Closely inspecting the proof in \cite[\S~5.4]{Kaplan2013a} (of \cite[(5.5)]{Kaplan2013a}, and see also
the top of p.~419 of \cite{Kaplan2015}, there $\beta^2=2\gamma$), we see that the only change is to \cite[Claim~5.6]{Kaplan2013a}
(this claim appeared as Claim~7.13 in \cite{Kaplan2015a} where it was proved in detail, but we reproduce the argument below), and we can observe the difference already when $n=1$. Thus we argue for $n=1$ (the extension to $n>1$ is straightforward).
We introduce the necessary notation from \cite[\S~5.4]{Kaplan2013a}.
Consider the subgroup $V_l''$ of $U_l$ defined by
\begin{align*}
V_l''=\left\{\begin{pmatrix}
1&0&0&0&v_4&0\\
&I_{l-2}&0&v_3&v_5&v_4'\\
&&1&0&v_3'&0\\
&&&1&0&0\\
&&&&I_{l-2}&0\\
&&&&&1\\
\end{pmatrix}\in U_l\right\}.
\end{align*}
Put $w'=\left(\begin{smallmatrix}&I_{l-1}\\1&\end{smallmatrix}\right)
\left(\begin{smallmatrix}I_{l-1}\\&4\end{smallmatrix}\right)$. Let $\varphi_{\zeta}$ belong to the space of $\Ind_{\overline{Q}_l}^{\SO_{2l}}(|\det|^{-\zeta}\sigma)$, where $\zeta$ is an auxiliary complex parameter ($\Re(\zeta)\gg0$) and $\sigma$ is realized in its Whittaker model with respect to the subgroup of upper triangular unipotent matrices in $\GL_l$ and character $z\mapsto \psi^{-1}(\sum_{i=1}^{l-1}z_{i,i+1})$.
Consider the function
\begin{align*}
F(h)=\int_{V_{l}''}\varphi_{\zeta}(v''w^{l,1}h,w')\psi_{\gamma}(v'')\,dv'',\qquad h\in\SO_3.
\end{align*}
We show $F(uh)=\psi^{-1}(u_{1,2})F(h)$, as opposed to \cite[Claim~5.6]{Kaplan2013a} where the claim was
$F(uh)=\psi^{-1}(\tfrac2\beta u_{1,2})F(h)$ under the assumption $2\gamma=\beta^2$ (see Remark~\ref{rem beta gamma} below).
For
\begin{align*}
u=\begin{pmatrix}
    1 &  &  \\
    x & 1 &  \\
    -\tfrac\gamma4x^2 & -\tfrac\gamma2x& 1
  \end{pmatrix},
\quad {}^{(w^{l,1})^{-1}}u=w^{l,1}u(w^{l,1})^{-1}=
\begin{pmatrix}
1 &  & & &\\
 & I_{l-2} & & &\\
\tfrac14x & & 1& & \\
\gamma x & & & 1 & \\
& & &  & I_{l-2} \\
-\tfrac\gamma4x^2 & & -\gamma x &  -\tfrac14x &  & 1
\end{pmatrix}.
\end{align*}
Then for $v''\in V_l''$, $({}^{(w^{l,n})^{-1}}u)^{-1}v''({}^{(w^{l,n})^{-1}}u)=b_uv_u$
where $b_u$ is the image in $M_{l}$ of
\begin{align*}
\left(\begin{array}{ccc} 1&&\\
\gamma x v_3-\tfrac\gamma4x^2v_4'&I_{l-2}&-\gamma x
v_4'\\&&1\\\end{array}\right)\in\GL_{l}
\end{align*}
and
\begin{align*}
v_u=\left(\begin{array}{cccccc}
I_{n}&0&0&0&v_4&0\\
&I_{l-n-1}&0&v_3-\tfrac14xv_4'&v_5+\ldots&v_4'\\
&&1&0&v_3'-\tfrac14xv_4&0\\
&&&1&0&0\\
&&&&I_{l-n-1}&0\\
&&&&&I_{n}\\
\end{array}\right)\in V_{l}''.
\end{align*}
It follows that
\begin{align*}
F(uh)=\int_{V_{l}''}\varphi_{\zeta}(({}^{(w^{l,n})^{-1}}u)b_uv_uw^{l,n}h,
w')\psi_{\gamma}(v'')\,dv''.
\end{align*}
Now on the one hand, changing variables in $v_u$ removes the
dependence on $u$ and changes $\psi_{\gamma}(v'')=\psi(-\gamma
(v_3)_{l-2})$ to
$\psi_{\gamma}(v'')\psi(-\tfrac\gamma4 x(v_4')_{l-2})$.
On the other hand, for any $h\in\SO_3$,
\begin{align*}
\varphi_{\zeta}(({}^{(w^{l,n})^{-1}}u)b_uh,
w')
=\psi^{-1}(-\tfrac\gamma4 x(v_4')_{l-2}+x)\varphi_{\zeta}(h,
w').
\end{align*}
We conclude $F(uh)=\psi^{-1}(x)F(h)=\psi^{-1}(u_{1,2})F(h)$ (cf. \cite[Claim~5.6]{Kaplan2013a}).
Plugging this result into \cite[\S~5.4, p.~340]{Kaplan2013a},
rewriting the $du$-integration over $U_{\SO_3}$ and changing variables $x\mapsto \tfrac2\gamma x$ we obtain the analogue of
\textit{loc. cit.} (5.22):
\begin{align}\label{int:1var k=l>n+1 before applying Shahidi local coefficient}
&|\tfrac2\gamma|\int_{\overline{U}_{\SO_3}\backslash \SO_3}F(h)
\left(\int_{U_{\SO_3}}
f_s(w_1uw_1^{-1}h,1)\psi^{-1}(\tfrac2\gamma u_{1,2})\,du\right)\,dh.
\end{align}
Now applying
\eqref{shahidigammafactor} we obtain the analogue of \textit{loc. cit.} (5.23): as meromorphic continuations
\eqref{int:1var k=l>n+1 before applying Shahidi local
coefficient} equals
\begin{align}\label{int:1var k=l>n+1 after applying Shahidi local coefficient}
&c_{\tau,\beta}|\tfrac2\gamma|\int_{\overline{U}_{\SO_3} \backslash \SO_3}F(h)
\left(\int_{U_{\SO_3}}
M^*(\tau,s)f_s(w_1u
w_1^{-1}h,1)\psi^{-1}(\tfrac2\gamma u_{1,2})\,du\right)\,dh.
\end{align}
Here
$c_{\tau,\beta}=|\tfrac2\gamma|^{2s-1}\tau^2(\tfrac2\gamma)=\tau^2(2)|2|^{2s-1}c(s,l,1,\gamma)$
is calculated by substituting $y\cdot f_s$ for $f_s$ in
\eqref{shahidigammafactor} where $y=\diag(\tfrac2\gamma,1,\tfrac\gamma2)w_1^{-1}h$.
Note that the extra factor $|\tfrac2\gamma|$ in \eqref{int:1var k=l>n+1 after applying Shahidi local coefficient} will be canceled when we proceed as in \cite[\S~5.4]{Kaplan2013a} and rewrite the $du$-integration over $\overline{U}_{\SO_3}$ again. Identity~\eqref{eq:multiplicativity helper results first var even orthogonal} now follows as in
\textit{loc. cit.}
\end{proof}
\begin{rmk}\label{rem beta gamma}
Even if $2\gamma=\beta^2$ here as well, for some $\beta$, $\SO_3$ is still defined differently, so one can not expect to reproduce the formula of \textit{loc. cit.} here unless $\gamma=2$, then $J_{3,\gamma}=J_3$, and if $\beta=2$ the embedding matches with our embedding. Then indeed $\tfrac2\beta=1$.
\end{rmk}

\def\cprime{$'$} \def\cprime{$'$} \def\cprime{$'$}

\end{document}